\documentclass[final, 3p, 12pt]{elsarticle}
\usepackage{setspace}
\usepackage{amssymb}
\usepackage{amsmath}
\usepackage{amsthm}
\usepackage{graphicx}

 \usepackage{mathptmx}      
\usepackage{latexsym}
 \usepackage{bm}
  \usepackage{textcomp} 

\usepackage{float}
\usepackage{longtable}
\journal{LMS Journal of  Computation and Mathematics}
\newtheorem{thm}{Theorem}

\newtheorem{lem}[thm]{Lemma}
\newtheorem{defn}[thm]{Definition}

\begin{document}
\newcommand{\etal}[1]{\emph{et al.}#1}
\newcommand{\oankm}{\text{OA}(N,k-1,s,t)}
\newcommand{\oan}{\text{OA}(N,k,s,t)}\newcommand{\oano}{\text{OA}(N,k_0,s,t)}
\newcommand{\oank}{\text{OA}(N,k+1,s,t)}
\newcommand{\oal}{\text{OA}(\lambda s^t,k,s,t)}
\newcommand{\fnst}{f(N,s,t)}
\newcommand{\GAbBdc}{G(\mathbf{A},\mathbf{b},\mathbf{B},\mathbf{d},\mathbf{c})}
\newcommand{\GAbc}{G(\mathbf{A},\mathbf{b},\mathbf{c})}
\newcommand{\fnt}{f(N,2,t)}
\newcommand{\ftnt}{f(2N,2,t+1)}
\newcommand{\fcan}{f^{\text{CA}}_{\lambda}(N,s,t)}
\newcommand{\fpan}{f^{\text{PA}}_{\lambda}(N,s,t)}
\newcommand{\fcant}{f^{\text{CA}}_{\lambda}(N,2,t)}
\newcommand{\fpant}{f^{\text{PA}}_{\lambda}(N,2,t)}

\newcommand{\kin}{k_{\mathrm{input}}}
\newcommand{\kout}{k_{\mathrm{out}}}
\newcommand{\oatntmm}{\text{OA}(2N,k,2,t+1)}
\newcommand{\oatntm}{\text{OA}(2N,k+1,2,t+1)}
\newcommand{\oantf}{\text{OA}(2N,k+1,2,t+1)}
\newcommand{\oantff}{\text{OA}(2N,k+2,2,t+1)}
\newcommand{\oanst}{\text{OA}(N,k,s,2)}
\newcommand{\oantm}{\text{OA}(N,k-1,2,t)}
\newcommand{\oant}{\text{OA}(N,k,2,t)}
\newcommand{\oantt}{\text{OA}(N,k+1,2,t)}
\newcommand{\oalt}{\text{OA}(\lambda 2^t,k,2,t)}
\newcommand{\fxoptt}{f(\X^{opt}_{N})}
\newcommand{\oam}{\text{OA}(N,s^{k_1}_1s^{k_2}_2\cdots s^{k_v},t)}
\newcommand{\yoptt}{\Y^{opt}_{N,k}}
\newcommand{\xoptt}{\X^{opt}_{N,k}}
\newcommand{\x}{\mathbf{x}}
\newcommand{\M}{\mathbf{M}}
\newcommand{\N}{\mathbf{N}}
\newcommand{\B}{\mathbf{B}}
\newcommand{\HH}{\mathbf{H}}
\newcommand{\I}{\mathbf{I}}
\newcommand{\1}{\mathbf{1}}
\newcommand{\zz}{\mathbf{0}}
\newcommand{\aaa}{\mathbf{a}}
\newcommand{\bb}{\mathbf{b}}
\newcommand{\cc}{\mathbf{c}}
\newcommand{\uu}{\mathbf{u}}
\newcommand{\rrr}{\mathbf{r}}
\newcommand{\xx}{\mathbf{x}}
\newcommand{\y}{\mathbf{y}}
\newcommand{\J}{\mathbf{J}}
\newcommand{\dd}{\mathbf{d}}
\newcommand{\vvec}{\mathbf{vec}}
\newcommand{\Y}{\mathbf{Y}}
\newcommand{\X}{\mathbf{X}}
\newcommand{\D}{\mathbf{D}}
\newcommand{\A}{\mathbf{A}}\newcommand{\tA}{\tilde{\mathbf{A}}}
\newcommand{\XX}{\mathbf{X}^{\top}\mathbf{X}}
\newcommand{\XT}{\mathbf{X}^{\top}}
\newcommand{\bs}{\mathbf}
\newcommand{\0}{\mathbf{0}}
\newcommand{\orr}{\text{\ or \ }}
\newcommand{\md}[2][4]{#2\ \ (\mathrm{mod}\ \ #1)}
\newcommand{\rr}{\mbox{I\!\!\,R}}
\newcommand{\es}{\mathrm{E}(s^2)}
\newcommand{\smax}{s_{\mathrm{max}}}
\newcommand{\freq}{f_{\smax}}
\newcommand{\dfk}[1]{\Delta f^k(#1)}
\newcommand{\Z}{\mathbb{Z}}
\newcommand{\mth}[1]{\multicolumn{3}{c|}{#1}}
\newcommand{\mtw}[1]{\multicolumn{2}{c|}{#1}}
\newcommand{\GWP}{\mathrm{GWP}}
\begin{frontmatter}
\title{The linear programming relaxation permutation symmetry group of an orthogonal array defining integer linear program}
\author[AFITO]{David M. Arquette}
\ead{David.Arquette@afit.edu}
\address[AFITO]{Department of Mathematics and Statistics, Air Force Institute of Technology,\\Wright-Patterson Air Force Base, Ohio 45433, USA}
\author[AFITT]{Dursun A. Bulutoglu\corref{cor1}}
\ead{dursun.bulutoglu@gmail.com}
\address[AFITT]{Department of Mathematics and Statistics, Air Force Institute of Technology,\\Wright-Patterson Air Force Base, Ohio 45433, USA}
\cortext[cor1]{Corresponding author}
\begin{abstract}
  There is always a natural embedding of $S_s\wr S_k$ into the linear programming (LP) relaxation permutation symmetry group
of an orthogonal array  integer linear programming (ILP) formulation with equality constraints.
  The point of this paper is to prove that in the $2$ level, strength $1$ case  the LP relaxation permutation symmetry group of this formulation is isomorphic to $S_2\wr S_k$ for all $k$, and in the $2$ level, strength $2$ case it is isomorphic to $S_2^k\rtimes S_{k+1}$ for $k\geq 4$. The strength $2$ result reveals previously unknown  permutation symmetries that can not be captured by the natural embedding of   $S_2\wr S_k$.
We also conjecture a complete characterization of the LP relaxation permutation symmetry group of the ILP formulation.
\end{abstract}
\end{frontmatter}
\section{Introduction}
 An $N$ run, $k$ factor, $s$ level factorial design is an $N\times k$ array where each column takes values from the set $\{0,1,\dots,s-1\}$.   If one design can be obtained from another by permuting runs (rows), factors (columns), or levels, then the two designs are said to be \textit{isomorphic}.    An orthogonal array $\oan$ is an  $N$  run,  $k$  factor,  $s$  level design where every  $t$-tuple appears within each combination of  $t$  columns $N/s^t=\lambda$ times. The values $t$ and $\lambda$ are called the \textit{ strength} and the \textit{index} of the $\oan$.  $\oan$ can be enumerated up to isomorphism by finding all non-isomorphic solutions of certain (ILP)s, see~\cite{Appa2006,Bulutoglu2008,Bulutoglu2016,Geyer2014,Rosenberg1995}. An ILP searches for the solution vector $\x$ 
to the following problem
\begin{equation}\label{eqn:ILP}
\min \ \mathbf{c}^{\top} \x:\ \mathbf{A} \mathbf{x}= \mathbf{b},\ \mathbf{B} \mathbf{x} \leq \mathbf{d}, \ \mathbf{x} \in \mathbb{Z}^n,
\end{equation}
where $\A$ and $\B$ are the equality and inequality constraint matrices, $\cc^{\top}\x$ is the objective function, and $\x$ is the integer solution vector.
The permutation symmetry group of an ILP is the set of all permutations of its variables that do not change the feasibility and optimality of its solutions.  More formally, Margot~\cite{Margot2010} defined the \textit{permutation symmetry group},  $G$, of an ILP to be
\begin{equation*}
G = \{\pi \in S_n ~|~ \mathbf{c}^{\top} \mathbf{x} = \mathbf{c}^{\top} \pi(\mathbf{x}) \text{ and } \pi(\mathbf{x}) \in \mathcal{F} \text{ } \forall~ \mathbf{x} \in \mathcal{F}\},
\end{equation*}
where  $\mathcal{F}$  is the set of all feasible solutions.  Symmetric ILPs can arise from a variety of problem formulations, see \cite{Liberti2012}.  In particular, ILPs for enumerating orthogonal arrays are highly symmetric. \par

Optimal solutions to ILPs are commonly found with a branch-and-bound or a branch-and-cut algorithm.  In the case of symmetric ILPs, many of the subproblems in an enumeration tree are isomorphic.  As a result, a drastic amount of computational time is wasted on solving identical problems.  As a remedy, Margot~\cite{Margot2003a,Margot2003b,Margot2007} developed a solver that is able to decrease and potentially eliminate such redundant computations by exploiting a subgroup of an ILP's permutation symmetry group when pruning an enumeration tree.  Exploiting larger subgroups results in increased reductions in redundant computations, and the greatest reduction is attained if the full permutation symmetry group of the ILP is exploited.  Hence, it is desirable to find larger subgroups of the permutation symmetry group of an ILP. \par

 Let $\A(\pi, \sigma)$ be the matrix obtained from $\A$ by permuting its columns by $\pi$ and its rows by $\sigma$.
Then for ILP (\ref{eqn:ILP})  \[G(\A,\bb,\B,\dd,\cc)= \left\{\pi \  | \ \pi(\cc)=\cc, \ \exists \ \sigma \ \mbox{with}\ \A(\pi, \sigma)=\A,\ \B(\pi, \sigma)=\B,
\   \sigma\left[{\bb\atop \dd}\right]=\left[{\bb\atop \dd}\right]\right\}\]
 is defined to be the {\em formulation symmetry group}.
Clearly,  $\GAbBdc \subseteq G$, and computational experiments suggest that solution times can be improved by several orders of magnitude if Margot's solver is used with  $\GAbBdc$  on ILPs with large $|\GAbBdc|$, 
see~\cite{Margot2010}.  \par

 The \textit {LP relaxation permutation symmetry group},  $G^{\rm LP}$  is the set of all permutations of variables that send LP feasible points to LP feasible points with the same objective function value. Clearly,  $G^{\rm LP}$ contains the formulation symmetry group. The  formulation symmetry groups of two different ILP formulations with the same set of variables, LP relaxation feasible set and objective function can be different. Hence, the formulation symmetry group
   does not always capture all the permutation symmetries in the LP relaxation of the ILP.   On the other hand, assuming that the set of variables in the ILP is fixed, $G^{\rm LP}$ depends only on the LP relaxation feasible set, and captures all the symmetries therein. 
For an LP relaxation without equality constraints, where each constraint in  $\mathbf{A} \mathbf{x} \leq \mathbf{b}$  is a facet (non-redundant),  $G^{\rm LP} = \GAbc$. 
\par 

Since the permutation symmetry group  $G$ is determined by the ILP's feasible set, identifying all symmetries in an ILP is difficult.  Margot~\cite{Margot2010} proved that deciding if  $G = S_n$  is NP-Complete.  Therefore, finding  $G$  for any given ILP is NP-Hard.  In order to find most of the symmetries in an ILP, one can simply find the permutation symmetry group of the LP relaxation.   Let  $\mathcal{F}$  and  $\mathcal{F}^{LP}$  be the sets of feasible solutions of an ILP and its LP relaxation, respectfully. Then $\mathcal{F} \subseteq \mathcal{F}^{LP}$ implies that
\begin{equation}\label{eqn:contain}
\GAbBdc \subseteq G^{\rm LP} \subseteq G
\end{equation}
 and there are examples where one or both of the containments in (\ref{eqn:contain}) are strict, see Geyer \cite{Geyer2014b}. \par

Let  the OA$(N,k,s,t)$ defining ILP (3) of Theorem 2 in Bulutoglu and Margot~\cite{Bulutoglu2008}  be called  the Bulutoglu and Margot formulation.
 Geyer~\cite{Geyer2014b} developed an algorithm for finding  $G^{\rm LP}$  of  a generic LP with equality constraints.  Furthermore, Geyer~\cite{Geyer2014b} observed that  $\GAbBdc \subsetneq G^{\rm LP}$  for  the Bulutoglu and Margot formulation  for finding OA$(N, k, 2, t)$  when  $t$  is even.  Geyer, Bulutoglu, and Rosenberg~\cite{Geyer2014} proved that $\GAbBdc$ is isomorphic to $S_s \wr S_k$  for this formulation.  They also showed  $G^{\rm LP}\cong S_{(s-1)}\wr S_k$  for a formulation without equality constraints.  For the strength $1$ and $2$ cases we prove the  computational observations in~\cite{Geyer2014b}  by explicitly finding the  $G^{\rm LP}$  of a different ILP formulation with the same set of variables and inequality constraints, but different equality constraints.  The equality constraints of this ILP formulation are linear combinations of those of the Bulutoglu and Margot formulation.  Furthermore, both ILPs have the same number of non-redundant equality constraints.  Hence, one can go back and forth between the two ILP formulations by applying a sequence of row operations to the equality constraints of each.  This implies that the feasible sets of the LP relaxations of these two ILP formulations are the same, so  their $G^{\rm LP}$s  must also be the same. The new ILP formulation makes it possible to find the $G^{\rm LP}$ of the Bulutoglu and Margot formulation by drastically simplifying this problem. \par

The new ILP formulation we use stems from the concept of J-characteristics. Let $\mathbf{D}$ be a $2$ level, $N$  run (row), $k$  factor (column) design with  levels from $\{1,-1\}$.
 Let the frequency vector,  $\mathbf{f}$ of $\mathbf{D}$, have the frequency of each of the  $2^k$  possible factor level combinations as its entries.  Hence,  $\mathbf{f}$  determines  $\mathbf{D}$  up to reordering of its runs.  Then  the \textit{J-characteristics}
of   $\mathbf{D}$ are given by
\begin{equation*}
J_l = \sum_{i = 1}^N \prod_{j \in l} d_{ij}
\end{equation*}
for  $l \subseteq \mathbb{Z}_k$, where $d_{ij}$ is the $(i,j)^\text{th}$ entry of $\D$.  It is well known that  $\mathbf{D}$  is uniquely determined by its J-characteristics up to reordering of its runs; furthermore,  $\mathbf{D}$  is an orthogonal array of strength  $t$  if and only if  $J_l = 0$  for all  $l \subseteq \mathbb{Z}_k$  with  $|l| \leq t$ and  $l \neq \emptyset$, see~\cite{Stufken2007}.

\section{Preliminaries and the main results}
Let  $\mathbf{1}$  be the column vector of length  $2^k$  for which every entry is one.  For  $i = 1, \ldots, k$, let  $\mathbf{x}_i$  be the  $i^\text{th}$  column ($i^\text{th}$  main effect) of the  $k$ factor, $2$ level ($\pm1$) full factorial design (the design that contains each of the $2^k$ factor level combinations exactly once).  For distinct  $i_1, \ldots, i_j \in \{1, \ldots, k\}$ with $j\geq2$, let  $\mathbf{x}_{i_1, \ldots, i_j}$  represent the  $j$ factor interaction term given by the Hadamard product  $\mathbf{x}_{i_1} \odot \cdots \odot \mathbf{x}_{i_j}$, where the $p^\text{th}$ entry of $\mathbf{x}_{i_1} \odot \cdots \odot \mathbf{x}_{i_j}$ is
the product of the entries on the $p^\text{th}$ row of the matrix   $[\mathbf{x}_{i_1}, \cdots,  \mathbf{x}_{i_j}]$. \par
	Equation (5) in ~\cite{Stufken2007} and the fact that  $\mathbf{D}$  is an orthogonal array of strength  $t$  if and only if  $J_l = 0$  for all  $l \subseteq \mathbb{Z}_k$  with  $|l| \leq t$ and  $l \neq \emptyset$ give us the following $\oant$ defining ILP
\begin{equation}\label{eqn:Mf}
\min \ 0 : \ \mathbf{M} \mathbf{f} = \mathbf{J}, \  \mathbf{f} \geq \0, \ \mathbf{f}\in \mathbb{Z}^{2^k},
\end{equation}
 where  $\mathbf{M}$  is the  $\sum\limits_{i=0}^t \binom{k}{i}$  by  $2^k$  matrix
\begin{equation}\label{eqn:M}
\mathbf{M} = \begin{bmatrix}
\mathbf{1}^{\top} \\
\mathbf{x}_1^{\top} \\
\vdots \\
\mathbf{x}_k^{\top} \\
\mathbf{x}_{1, 2}^{\top} \\
\vdots \\
\mathbf{x}_{{k - t + 1}, \dots, k}^{\top} \end{bmatrix},
\end{equation}
$\mathbf{J}$  is the J-characteristic vector with entries  $J_l$  for  $|l| \leq t$, 
\begin{equation}\label{eqn:J}
\mathbf{J} = \begin{bmatrix}
N \\
0 \\
\vdots \\
0 \end{bmatrix},
\end{equation}
$\0$ is the all zeros vector, and  $\mathbf{f}$  is the frequency vector of a hypothetical OA$(N, k, 2, t)$.  

Our goal is to find the subgroup of the permutation group  $S_{2^k}$  that sends the LP relaxation feasible solutions  ($\mathbf{f} \in \mathbb{Q}_{\geq 0}^{2^k}$) of ILP (\ref{eqn:Mf}) to LP relaxation feasible solutions.   The equality constraints of this ILP are linear combinations of those of the Bulutoglu and Margot formulation.  Both ILPs have the same inequality constraints, and each ILP has  $\sum\limits_{i = 0}^t \binom{k}{i}$  non-redundant equality constraints.  Hence, both ILPs have the same LP relaxation feasible set, and this implies that both have the same LP relaxation permutation symmetry group.  From this point on, we shall refer to this group as  $G^{\rm LP}$.
The main results of this paper are the following theorems.
\begin{thm} \label{thm:easy}
For an OA$(N,k,2,1)$
\begin{equation*}
G^{\rm LP} \cong S_2 \wr S_k.
\end{equation*}
\end{thm}

\begin{thm}\label{thm:main}
 For an OA$(N,k,2,2)$ with   $k \geq 4$
\begin{equation*}
G^{\rm LP} \cong S_2^k \rtimes S_{k+1}.
\end{equation*}
\end{thm}
The proof of Theorem \ref{thm:main} requires more work than the proof of Theorem \ref{thm:easy}.
The following definition and lemma are necessary to prove both theorems.
\begin{defn}
Let $\mathbf{M}$ be a matrix and ${\rm Row}(\mathbf{M})$ be the row space of $\mathbf{M}$. Then the set of all invertible linear transformations $\mathbf{T}$ such that $\mathbf{T}(\mathbf{v})\in {\rm Row}(\mathbf{M})$ for all $\mathbf{v}\in {\rm Row}(\mathbf{M})$ is called the {\textit automorphism group} of ${\rm Row}(\mathbf{M})$ and denoted by ${\rm Aut}({\rm Row}(\mathbf{M}))$.
\end{defn}
\begin{lem}\label{thm:AutRowM}
Let $\mathbf{M}$ and $\mathbf{J}$ be as in equations  (\ref{eqn:M}) and (\ref{eqn:J})  respectively. Then $$G^{\rm LP}={\rm Aut}({\rm Row}(\mathbf{M})) \cap S_{2^k}.$$  
\end{lem}

\begin{proof} Let LP (\ref{eqn:Mf}) be the LP relaxation of ILP (\ref{eqn:Mf}).
Observe that
\begin{equation*}
\mathbf{f}^* = \begin{bmatrix}
\frac{N}{2^k} \\
\vdots \\
\frac{N}{2^k} \end{bmatrix}
\end{equation*}
is a particular solution to LP (\ref{eqn:Mf}).  Hence, every solution  $\mathbf{f}$ to the LP (\ref{eqn:Mf})  can be written in the form  $\mathbf{f}^* + \mathbf{f}'$,  where  $\mathbf{f}'$ comes from the null space of $\mathbf{M}$.  Let  $g \in G^{\rm LP}$  be arbitrary.  Then  $g(\mathbf{f})$  is a solution to LP (\ref{eqn:Mf}).  That is,
\begin{equation*}
\mathbf{M} g(\mathbf{f}) = \mathbf{M} g(\mathbf{f}^* + \mathbf{f}')=\mathbf{J}.
\end{equation*}
Because  $g \in G^{\rm LP} \leq S_{2^k}$,
\begin{equation*}
\mathbf{M} [g(\mathbf{f}^*) + g(\mathbf{f}')] = \mathbf{M} [\mathbf{f}^* + g(\mathbf{f}')] = \mathbf{J},
\end{equation*}
and thus
\begin{equation*}
\mathbf{M} \mathbf{f}^* + \mathbf{M} g(\mathbf{f}') = \mathbf{J} + \mathbf{M} g(\mathbf{f}') = \mathbf{J}.
\end{equation*}
Therefore,
\begin{equation*}
\mathbf{M} g(\mathbf{f}') = \mathbf{0},
\end{equation*}
so we see that  $g(\mathbf{f}') \in {\rm Null}(\mathbf{M})$  which means  $g$  must preserve  ${\rm Null}(\mathbf{M})$.  Because  $g \in G^{\rm LP}$  is arbitrary,   $G^{\rm LP} \leq {\rm Aut}({\rm Null}(\mathbf{M})) \cap S_{2^k}$. \par
Now let  $h \in {\rm Aut}({\rm Null}(\mathbf{M})) \cap S_{2^k}$  be arbitrary.  Then
\begin{eqnarray*}
\mathbf{M} h(\mathbf{f}) & = & \mathbf{M} h(\mathbf{f}^* + \mathbf{f}') \\
& = & \mathbf{M} \mathbf{f}^* + \mathbf{M} h(\mathbf{f}') \\
& = & \mathbf{J}.
\end{eqnarray*}
Also, $$\mathbf{f}\geq\0 \quad \Rightarrow \quad h(\mathbf{f})\geq\0.$$
Hence,  $h \in G^{\rm LP}$, and because  $h$  is arbitrary,  ${\rm Aut}({\rm Null}(\mathbf{M})) \cap S_{2^k} \leq G^{\rm LP}$.  As  ${\rm Aut}({\rm Null}(\mathbf{M})) = {\rm Aut}({\rm Row}(\mathbf{M}))$, we conclude that  $G^{\rm LP} = {\rm Aut}({\rm Null}(\mathbf{M})) \cap S_{2^k} = {\rm Aut}({\rm Row}(\mathbf{M})) \cap S_{2^k}$.
\end{proof}

\section{The proof of the strength one case}

Throughout this section we assume that $t=1$  in ILP (\ref{eqn:Mf}).  Let
$\mathcal{B} = \{\mathbf{1},  \mathbf{x}_1, \ldots, \mathbf{x}_k\}.$
Then $\mathcal{B}$  is an orthogonal basis for  ${\rm Row}(\mathbf{M})$.  For all  $g \in G^{\rm LP}$,   $g(\mathcal{B})$  must also be an orthogonal basis for  ${\rm Row}(\mathbf{M})$  because,  $g \in S_{2^k}$, and elements of  $S_{2^k}$  preserve angles.  Furthermore, for every  $\mathbf{x} \in \mathcal{B}$, $g(\mathbf{x})$  can be written uniquely as a linear combination of the elements of  $\mathcal{B}$.  That is,
\begin{equation}\label{eqn:glincomb}
g(\mathbf{x}) = \lambda_0 \mathbf{1} + \lambda_1 \mathbf{x}_1 + \cdots + \lambda_k \mathbf{x}_k.
\end{equation}

\begin{lem}\label{lem:lam101}
Let  $\mathbf{x} \in \mathcal{B}$.  If  $\mathbf{x} = \mathbf{1}$  in (\ref{eqn:glincomb}), then  $\lambda_0 = 1$, and  $\lambda_i = 0$  for  $i = 1, \ldots, k$.  Otherwise,  $\lambda_0 = 0$.
\end{lem}

\begin{proof}
Suppose  $\mathbf{x} = \mathbf{1}$.  Because  $g \in G^{\rm LP} \leq S_{2^k}$,  $g(\mathbf{1}) = \mathbf{1}$  which uniquely satisfies (\ref{eqn:glincomb}).  For  $i = 1, \ldots, k$, $g(\mathbf{x}_i)$  must be orthogonal to  $g(\mathbf{1}) = \mathbf{1}$, so  $\lambda_0 = 0$  whenever  $\mathbf{x} \neq \mathbf{1}$.
\end{proof}

\begin{lem}\label{lem:xisigi}
Let $\{\mathbf{x}_1, \ldots, \mathbf{x}_k\}$ be the columns of a full factorial  $2^k$  design appearing in $\mathcal{B}$. If  $\{\mathbf{x}_1', \ldots, \mathbf{x}_k'\}$  is obtained from   $\{\mathbf{x}_1, \ldots, \mathbf{x}_k\}$, by permuting rows, then there exists a permutation  $\sigma \in S_k$  such that for all  $i \in \{1, \ldots, k\}$  satisfying  $\mathbf{x}_i' \in span(\mathbf{x}_1, \ldots, \mathbf{x}_k)$,  $\mathbf{x}_i' = \pm\mathbf{x}_{\sigma(i)}$.
\end{lem}

\begin{proof}
  Suppose  $\mathbf{x}_i' \in span(\mathbf{x}_1, \ldots, \mathbf{x}_k)$.  Then  $\mathbf{x}_i' = \lambda_1 \mathbf{x}_1 + \cdots + \lambda_k \mathbf{x}_k$, and we have the system of equations
\begin{equation*}
\begin{array}{cccccc}
\lambda_1 & + & \cdots & + \lambda_k & = & \pm1 \\
\lambda_1 & + & \cdots & - \lambda_k & = & \pm1 \\
\vdots &   & \ddots & \vdots &   & \vdots \\
-\lambda_1 & - & \cdots & - \lambda_k & = & \pm1. \end{array}
\end{equation*}
Subtracting the second equation from the first equation gives  $\lambda_k \in \{0, \pm1\}$.  Choosing other pairs of equations similarly yields  $\lambda_j \in \{0, \pm1\}$  for  $j = 1, \ldots, k$.  Because  $\mathcal{B}$  is an orthogonal set, the Pythagorean theorem gives
\[
\|\mathbf{x}_i'\|^2  =  \sum_{j=1}^k \|\lambda_j \mathbf{x}_j\|^2 = \sum_{j=1}^k \lambda_j^2 \|\mathbf{x}_j\|^2.\]
Since row permutations are norm-preserving,  $\|\mathbf{x}_i'\|^2 = 2^k = \|\mathbf{x}_j\|^2$  for  $j = 1, \ldots, k$.  Thus,
$\sum_{j=1}^k \lambda_j^2 = 1.$
Since $\lambda_j \in \{0, \pm1\}$  for  $j = 1, \ldots, k$, there is exactly one nonzero  $\lambda_j \in \{\pm1\}$, and  $\mathbf{x}_i' = \pm\mathbf{x}_j$. Row permutations also preserve orthogonality, so for every distinct  $i \in \{1, \ldots, k\}$  such that  $\mathbf{x}_i' \in span(\mathbf{x}_1, \ldots, \mathbf{x}_k)$, there is a unique  $j \in \{1, \ldots, k\}$  satisfying  $\mathbf{x}_i' = \pm\mathbf{x}_j$.  Thus, there exists a permutation  $\sigma \in S_k$  such that  $\mathbf{x}_i' = \pm\mathbf{x}_{\sigma(i)}$  for all  $i \in \{1, \ldots, k\}$  such that  $\mathbf{x}_i' \in span(\mathbf{x}_1, \ldots, \mathbf{x}_k)$.
\end{proof}

\begin{lem}\label{lem:ubnd1G}
\begin{equation*}
|G^{\rm LP}| \leq 2^k k!.
\end{equation*}
\end{lem}

\begin{proof}
Let  $g \in G^{\rm LP}$  be arbitrary.  From Lemma \ref{lem:lam101}, we have that  $g(\mathbf{1}) = \mathbf{1}$, and  $g(\mathbf{x}_i) = \lambda_1 \mathbf{x}_1 + \cdots + \lambda_k \mathbf{x}_k$  for  $i = 1, \ldots, k$.  Now by Lemma \ref{lem:xisigi},  there exists a permutation  $\sigma \in S_k$  such that  $\mathbf{x}_i' = \pm\mathbf{x}_{\sigma(i)}$  for  $i = 1, \ldots, k$.  That is,  $g$  is essentially a signed permutation of the  $k$  main effects, so  $g$  is one of at most  $2^k k!$  elements in  $G^{\rm LP}$.
\end{proof}

Next, by finding a subgroup of  $G^{\rm LP}$  that attains the upper bound on size, and using the semidirect product and the wreath product  we determine the size and the structure of  $G^{\rm LP}$.   For the definitions of the semidirect product, the wreath product, and the base of a wreath product see Rotman~\cite{Rotman1994}.

In the factorial design setting,  $S_k$  is the permutation group of  $k$  factors.  The multiplicative group  $\{\pm1\}$  that multiplies a column is isomorphic to  $S_2$.  Naturally,  $S_2^k$  is the direct product of  $k$  copies of  $S_2$.  We now see that  $S_2^k$  is the base of  $S_2 \wr S_k = S_2^k \rtimes S_k$ and
 $S_2^k \trianglelefteq S_2 \wr S_k$.

This wreath product is the set of all signed permutations of  $\mathbf{x}_i$  for  $i = 1, \ldots, k$  from the full factorial  $2^k$  design, where  $\mathbf{x}_1^{\top}, \ldots, \mathbf{x}_k^{\top}$  constitute rows of  $\mathbf{M}$.  We shall see that this group is a subgroup of  $G^{\rm LP}$.  Hence, by Lemma \ref{lem:ubnd1G}, it is isomorphic to $G^{\rm LP}$.

\begin{proof}[{\bf Proof of Theorem \ref{thm:easy}}]
An arbitrary element of  $S_2 \wr S_k$ can be written in the form  $\phi \sigma$  where  $\phi \in S_2^k$  and  $\sigma \in S_k$.   Clearly, permuting the  $k$  rows  $\mathbf{x}_1^{\top}, \ldots, \mathbf{x}_k^{\top}$  of  $\mathbf{M}$ or  negating any subset  of these  $k$  rows will preserve the full factorial  $2^k$  design. Hence,  $\phi \sigma \in S_{2^k}$.  Furthermore, the signed permutation  $\phi \sigma$  preserves  ${\rm Row}(\mathbf{M})$, so  $\phi \sigma \in {\rm Aut}({\rm Row}(\mathbf{M}))$.  Because  $\phi \sigma \in S_2 \wr S_k$  was arbitrary, we get that an isomorphic copy of $S_2 \wr S_k$ is contained in $G^{\rm LP} = {\rm Aut}({\rm Row}(\mathbf{M})) \cap S_{2^k}$.  Finally,  $|S_2 \wr S_k| = |S_2^k| |S_k| = 2^k k!$  is the upper bound for  $|G^{\rm LP}|$, so  $G^{\rm LP}$  must be isomorphic to  $S_2 \wr S_k$.
\end{proof}

\section{The proof of the strength two case}

Throughout this section, we assume that $t=2$ in ILP (\ref{eqn:Mf}).  Let
$\mathcal{B} = \{\mathbf{1},  \mathbf{x}_1, \ldots, \mathbf{x}_k, \mathbf{x}_{1, 2}, \ldots, \mathbf{x}_{{k - 1}, k}\}.$
Then  $\mathcal{B}$  is an orthogonal basis for  ${\rm Row}(\mathbf{M})$, and for all  $g \in G^{\rm LP}\subseteq S_{2^k}$ the set $g(\mathcal{B})$  must also be an orthogonal basis for  ${\rm Row}(\mathbf{M})$.  Also, for every  $\mathbf{x} \in \mathcal{B}$,  $g(\mathbf{x})$  can be written uniquely as a linear combination of the elements of  $\mathcal{B}$.  In this case,
\begin{equation}\label{eqn:glincomb2}
g(\mathbf{x}) =
\lambda_0 \mathbf{1} + \lambda_1 \mathbf{x}_1 + \cdots + \lambda_k \mathbf{x}_k + \lambda_{1, 2} \mathbf{x}_{1, 2} + \cdots + \lambda_{k - 1, k} \mathbf{x}_{k - 1, k}.
\end{equation}
Similar to the strength one case above, we will arrive at the conclusion that for any  $\mathbf{x} \in \mathcal{B}$, every  $\lambda$  in (\ref{eqn:glincomb2}) must be zero except for one, which must have an absolute value of one.  The following several lemmas serve to lead us to this conclusion.

\begin{lem}\label{lem:lam1}
Let  $\mathbf{x} \in \mathcal{B}$.  If  $\mathbf{x} = \mathbf{1}$  in (\ref{eqn:glincomb2}), then  $\lambda_0 = 1$, and  $\lambda_i = 0$  for  $i = 1, \ldots, k, (1, 2), \ldots, (k - 1, k)$.  Otherwise,  $\lambda_0 = 0$.
\end{lem}

\begin{proof}
Suppose  $\mathbf{x} = \mathbf{1}$.  Because  $g \in G^{\rm LP} \leq S_{2^k}$,  $g(\mathbf{1}) = \mathbf{1}$  which uniquely satisfies (\ref{eqn:glincomb2}).  For  $i = 1, \ldots, k, (1, 2), \ldots, (k - 1, k)$,  $g(\mathbf{x}_i)$  must be orthogonal to  $g(\mathbf{1}) = \mathbf{1}$, so  $\lambda_0 = 0$  whenever  $\mathbf{x} \neq \mathbf{1}$.
\end{proof}

\begin{lem}\label{lem:0.51}
Let  $\mathbf{x} \in \mathcal{B}$.  If  $\mathbf{x} \neq \mathbf{1}$  in (\ref{eqn:glincomb2}), then  $\lambda_i \in \{0, \pm0.5, \pm1\}$  for  $i = 1, \ldots, k, (1, 2), \ldots, (k - 1, k)$.
\end{lem}

\begin{proof}
Suppose  $\mathbf{x} \neq \mathbf{1}$.  Then (\ref{eqn:glincomb2}) becomes  $g(\mathbf{x}) = \lambda_1 \mathbf{x}_1 + \cdots + \lambda_k \mathbf{x}_k + \lambda_{1, 2} \mathbf{x}_{1, 2} + \cdots + \lambda_{k - 1, k} \mathbf{x}_{k - 1, k}$.  Because these basis vectors are the columns of the full factorial  $2^k$  design and the corresponding $2$-factor interactions are obtained by taking the appropriate pairwise Hadamard products of the individual columns (main effects), we have the system of equations
\begin{equation*}
\begin{array}{ccccccccccccc}
\lambda_1 & + & \cdots & + \lambda_k & + \lambda_{1, 2} & + & \cdots & + \lambda_{1, k} & + & \cdots & + \lambda_{k - 1, k} & = & \pm1 \\
\vdots &   & \ddots & \vdots & \vdots &   & \ddots & \vdots &   & \ddots & \vdots &   & \vdots \\
\lambda_1 & - & \cdots & - \lambda_k & - \lambda_{1, 2} & - & \cdots & - \lambda_{1, k} & + & \cdots & + \lambda_{k - 1, k} & = & \pm1 \\
-\lambda_1 & + & \cdots & + \lambda_k & - \lambda_{1, 2} & - & \cdots & - \lambda_{1, k} & + & \cdots & + \lambda_{k - 1, k} & = & \pm1 \\
\vdots &   & \ddots & \vdots & \vdots &   & \ddots & \vdots &   & \ddots & \vdots &   & \vdots \\
-\lambda_1 & - & \cdots & - \lambda_k & + \lambda_{1, 2} & + & \cdots & + \lambda_{1, k} & + & \cdots & + \lambda_{k - 1, k} & = & \pm1. \end{array}
\end{equation*}
Subtracting the last equation from the first gives  $\lambda_1 + \cdots + \lambda_k \in \{0, \pm1\}$.  Taking the difference of the middle equations likewise provides  $\lambda_1 - \cdots - \lambda_k \in \{0, \pm1\}$.  Summing these two expressions results in the conclusion  $\lambda_1 \in \{0, \pm0.5, \pm1\}$.  Choosing other sets of equations similarly yields  $\lambda_i \in \{0, \pm0.5, \pm1\}$  for  $i = 1, \ldots, k, (1, 2), \ldots, (k - 1, k)$.
\end{proof}

\begin{lem}\label{lem:form}
Let  $\mathbf{x} \in \mathcal{B}$  and  $g \in G^{\rm LP}$ in (\ref{eqn:glincomb2}) such that  $\mathbf{x} \neq \mathbf{1}$.  Then either  $g(\mathbf{x}) = \pm\mathbf{x}_i$  for some  $i \in \{1, \ldots, k, (1, 2) , \ldots, (k - 1, k)\}$  or  $g(\mathbf{x}) = \pm0.5\mathbf{x}_a \pm0.5\mathbf{x}_b \pm0.5\mathbf{x}_c \pm0.5\mathbf{x}_d$  for some distinct $a, b, c, d \in \{1, \ldots, k, (1, 2), \ldots, (k - 1, k)\}$.
\end{lem}

\begin{proof}
Because  $\mathcal{B}$  is an orthogonal set, the Pythagorean theorem gives
\[
\|g(\mathbf{x})\|^2  =  \sum_i \|\lambda_i \mathbf{x}_i\|^2 = \sum_i \lambda_i^2 \|\mathbf{x}_i\|^2.\]
Each  $g \in S_{2^k}$  is norm-preserving, so  $\|g(\mathbf{x})\|^2 = 2^k = \|\mathbf{x}_i\|^2$  for  $i =  1, \ldots, k, (1, 2), \ldots, (k - 1, k)$.  Thus,
$\sum_i \lambda_i^2 = 1.$  Clearly, not every  $\lambda_i$  can be zero.  If  $\lambda_i \in \{\pm1\}$  for some  $i \in \{1, \ldots, k, (1, 2), \ldots, (k - 1, k)\}$, then  $\lambda_j = 0$  for all  $j \in \{1, \ldots, k, (1, 2), \ldots, (k - 1, k)\}$  such that  $i \neq j$.  Otherwise, there must be distinct  $a, b, c, d \in \{1, \ldots, k, (1, 2), \ldots, (k - 1, k)\}$  such that  $\lambda_a, \lambda_b, \lambda_c, \lambda_d \in \{\pm0.5\}$, and every other  $\lambda$  is zero.  That is, either  $g(\mathbf{x}) = \pm\mathbf{x}_i$  for some $i \in \{1, \ldots, k, (1, 2) , \ldots, (k - 1, k)\}$  or  $g(\mathbf{x}) = \pm0.5\mathbf{x}_a \pm0.5\mathbf{x}_b \pm0.5\mathbf{x}_c \pm0.5\mathbf{x}_d$  for some distinct  $a, b, c, d \in \{1, \ldots, k, (1, 2), \ldots, (k - 1, k)\}$.
\end{proof}

\begin{lem}\label{lem:finform}
Let  $\mathbf{x} \in \mathcal{B}$  and  $g \in G^{\rm LP}$.  If  $g(\mathbf{x})$  is of the second form given in Lemma \ref{lem:form}, then  $g(\mathbf{x}) = \pm0.5\mathbf{x}_{a, b} \pm0.5\mathbf{x}_{a, c} \pm0.5\mathbf{x}_b \pm0.5\mathbf{x}_c$  for some distinct  $a, b, c \in \{1, \ldots, k\}$.
\end{lem}

\begin{proof}
Suppose  $g(\mathbf{x}) = \pm0.5\mathbf{x}_a \pm0.5\mathbf{x}_b \pm0.5\mathbf{x}_c \pm0.5\mathbf{x}_d$  for some distinct  $a, b, c, d \in  \{1, \ldots, k, (1, 2)$ $, \ldots, (k - 1, k)\}$.  Clearly,  $\mathbf{x}_a$,  $\mathbf{x}_b$,  $\mathbf{x}_c$, and  $\mathbf{x}_d$  cannot all be main effects, for the full factorial design will ensure some entry of  $g(\mathbf{x})$  equals  $2 \notin \{\pm1\}$.  Therefore, at least one $2$-factor interaction must be present in the linear combination.  Because there are more such linear combinations than would be prudent to check manually, we took advantage of R software~\cite{R2013}.  The code used for this step is contained in the Appendix.  By creating every essentially unique linear combination containing at least one $2$-factor interaction term and checking whether each satisfies a basic requirement, we ruled out all possibilities except those of one particular form.  Specifically, by ruling out each linear combination where the minimum and maximum entries in the resulting vector are not  $-1$  and  $1$, respectively, we eliminated all linear combinations except those of the form  $\pm0.5\mathbf{x}_{a, b} \pm0.5\mathbf{x}_{a, c} \pm0.5\mathbf{x}_b \pm0.5\mathbf{x}_c$  for some distinct  $a, b, c \in \{1, \ldots, k\}$.
\end{proof}

It is clear that  $k \geq 3$  in order for the form in Lemma \ref{lem:finform} to be viable.

\begin{lem}\label{lem:cannotbe}
Let  $\mathbf{x} \in \mathcal{B}$  and  $g \in G^{\rm LP}$.  If for every  $i \in \{1, \ldots, k\}$,  $g(\mathbf{x}_i) \neq \pm0.5\mathbf{x}_{a, b} \pm0.5\mathbf{x}_{a, c} \pm0.5\mathbf{x}_b \pm0.5\mathbf{x}_c$  for some distinct  $a, b, c \in \{1, \ldots, k\}$, then  $g(\mathbf{x})$  cannot be of the form in Lemma \ref{lem:finform}.
\end{lem}

\begin{proof}
Recall by Lemma \ref{lem:lam1} that  $g(\mathbf{1}) = \mathbf{1}$.  Suppose that for every  $i \in \{1, \ldots, k\}$, we have  $g(\mathbf{x}_i) \neq \pm0.5\mathbf{x}_{a, b} \pm0.5\mathbf{x}_{a, c} \pm0.5\mathbf{x}_b \pm0.5\mathbf{x}_c$  for some distinct  $a, b, c \in \{1, \ldots, k\}$.  Then by Lemmas \ref{lem:form} and \ref{lem:finform},   for every  $i \in \{1, \ldots, k\}$, there exists some  $j \in \{1, \ldots, k, (1, 2), \ldots, (k - 1, k)\}$  such that  $g(\mathbf{x}_i) = \pm\mathbf{x}_j$.  Because  $g$  preserves Hadamard products, for every  $i \in \{(1, 2), \ldots, (k - 1, k)\}$, there exists some  $j \in \{1, \ldots, k, (1, 2), \ldots, (k - 1, k)\}$  such that  $g(\mathbf{x}_i) = \pm\mathbf{x}_j$.  Hence, for every  $\mathbf{x} \in \mathcal{B}$,  $g(\mathbf{x}) \neq \pm0.5\mathbf{x}_{a, b} \pm0.5\mathbf{x}_{a, c} \pm0.5\mathbf{x}_b \pm0.5\mathbf{x}_c$  for some distinct  $a, b, c \in \{1, \ldots, k\}$.
\end{proof}

\begin{lem} \label{lem:finalformforce}
Let  $g \in G^{\rm LP}$.  If for some  $i \in \{1, \ldots, k\}$,  $g(\mathbf{x}_i) = \pm0.5\mathbf{x}_{a, b} \pm0.5\mathbf{x}_{a, c} \pm0.5\mathbf{x}_b \pm0.5\mathbf{x}_c$  for some distinct  $a, b, c \in \{1, \ldots, k\}$, then there must exist some  $j \in \{1, \ldots, k\}$  with  $i \neq j$  such that  $g(\mathbf{x}_j) = \pm0.5\mathbf{x}_{a', b'} \pm0.5\mathbf{x}_{a', c'} \pm0.5\mathbf{x}_{b'} \pm0.5\mathbf{x}_{c'}$  for some distinct  $a', b', c' \in \{1, \ldots, k\}$.
\end{lem}

\begin{proof}
Suppose there exists some  $i \in \{1, \ldots, k\}$,  $g(\mathbf{x}_i) = \pm0.5\mathbf{x}_{a, b} \pm0.5\mathbf{x}_{a, c} \pm0.5\mathbf{x}_b \pm0.5\mathbf{x}_c$  for
 some distinct  $a, b, c \in \{1, \ldots, k\}$.  By way of contradiction, suppose there is no  $j \in \{1, \ldots, k\}$  with  $i \neq j$  such that  $g(\mathbf{x}_j) = \pm0.5\mathbf{x}_{a', b'} \pm0.5\mathbf{x}_{a', c'} \pm0.5\mathbf{x}_{b'} \pm0.5\mathbf{x}_{c'}$  for some distinct  $a', b', c' \in \{1, \ldots, k\}$.  Then by Lemma
   \ref{lem:0.51} for every  $j \in \{1, \ldots, k\}$  with  $i \neq j$, there exists some  $l \in \{1, \ldots, k, (1, 2), \ldots, (k - 1, k)\}$  such that  $g(\mathbf{x}_j) = \pm\mathbf{x}_l$.  Because  $g$  preserves Hadamard products,  $g(\mathbf{x}_i) \odot \pm\mathbf{x}_l$  must also take on a viable form, and this implies that  $\mathbf{x}_l \in \{\mathbf{x}_a, \mathbf{x}_{b, c}\}$.  There can only be one such  $\mathbf{x}_j$  because if there were more than one, their Hadamard product would be sent to something in  $\{\pm\mathbf{1}, \pm\mathbf{x}_{a, b, c}\}$.  But that means only two main effects ($\mathbf{x}_i$  and  $\mathbf{x}_j$) get sent to viable forms, which contradicts Lemma \ref{lem:form}.  Thus, there must exist some  $j \in \{1, \ldots, k\}$  with  $i \neq j$  such that  $g(\mathbf{x}_j) = \pm0.5\mathbf{x}_{a', b'} \pm0.5\mathbf{x}_{a', c'} \pm0.5\mathbf{x}_{b'} \pm0.5\mathbf{x}_{c'}$  for some distinct  $a', b', c' \in \{1, \ldots, k\}$.
\end{proof}

\begin{lem}\label{lem:makeup}
Let  $g \in G^{\rm LP}$.  If there exist distinct  $i, j \in \{1, \ldots, k\}$  such that  $g(\mathbf{x}_i) = \pm0.5\mathbf{x}_{a, b} \pm0.5\mathbf{x}_{a, c} \pm0.5\mathbf{x}_b \pm0.5\mathbf{x}_c$  for some distinct  $a, b, c \in \{1, \ldots, k\}$  and  $g(\mathbf{x}_j) = \pm0.5\mathbf{x}_{a', b'} \pm0.5\mathbf{x}_{a', c'} \pm0.5\mathbf{x}_{b'} \pm0.5\mathbf{x}_{c'}$  for some distinct  $a', b', c' \in \{1, \ldots, k\}$, then  $\{a, b, c\} = \{a', b', c'\}$.
\end{lem}

\begin{proof}
Suppose there exist distinct  $i, j \in \{1, \ldots, k\}$  such that  $g(\mathbf{x}_i) = \pm0.5\mathbf{x}_{a, b} \pm0.5\mathbf{x}_{a, c} \pm0.5\mathbf{x}_b \pm0.5\mathbf{x}_c$  for some distinct  $a, b, c \in \{1, \ldots, k\}$  and  $g(\mathbf{x}_j) = \pm0.5\mathbf{x}_{a', b'} \pm0.5\mathbf{x}_{a', c'} \pm0.5\mathbf{x}_{b'} \pm0.5\mathbf{x}_{c'}$  for some distinct  $a', b', c' \in \{1, \ldots, k\}$.  We proceed by way of contradiction and suppose that  $\{a, b, c\} \neq \{a', b', c'\}$.  That is,  $|\{a, b, c\} \cap \{a', b', c'\}| < 3$.  We observe that
\begin{equation}\label{eqn:gxst2}
\begin{array} {lllll}
g(\mathbf{x}_{i, j}) = & \pm0.25\mathbf{x}_{a, b, a', b'} & \pm0.25\mathbf{x}_{a, b, a', c'} & \pm0.25\mathbf{x}_{a, b, b'} & \pm0.25\mathbf{x}_{a, b, c'} \\
& \pm0.25\mathbf{x}_{a, c, a', b'} & \pm0.25\mathbf{x}_{a, c, a', c'} & \pm0.25\mathbf{x}_{a, c, b'} & \pm0.25\mathbf{x}_{a, c, c'} \\
& \pm0.25\mathbf{x}_{b, a', b'} & \pm0.25\mathbf{x}_{b, a', c'} & \pm0.25\mathbf{x}_{b, b'} & \pm0.25\mathbf{x}_{b, c'} \\
& \pm0.25\mathbf{x}_{c, a', b'} & \pm0.25\mathbf{x}_{c, a', c'} & \pm0.25\mathbf{x}_{c, b'} & \pm0.25\mathbf{x}_{c, c'}. \end{array}
\end{equation} \par
	(Case 1:   $|\{a, b, c\} \cap \{a', b', c'\}| = 0$)  (\ref{eqn:gxst2}) clearly is not of a valid form.  \par
	(Case 2:   $|\{a, b, c\} \cap \{a', b', c'\}| = 1$)  If  $a \neq a'$, $4$-factor interaction terms will remain in (\ref{eqn:gxst2}), so it will not be of a valid form.  Suppose  $a = a'$.  Even if the $3$-factor interaction terms were to cancel, the remaining $2$-factor interaction terms are insufficient for (\ref{eqn:gxst2}) to be of a valid form.  \par
	(Case 3:   $|\{a, b, c\} \cap \{a', b', c'\}| = 2$)  If  $a \neq a'$, at least one $4$-factor interaction term will remain in (\ref{eqn:gxst2}), so it will not be of a valid form.  Suppose  $a = a'$.  Without loss of generality, also suppose  $b = b'$.  Even if the $3$-factor interaction terms and the  $\mathbf{1}$  terms were to cancel, the remaining $2$-factor interaction terms and main effect terms are insufficient for (\ref{eqn:gxst2}) to be of a valid form.
\end{proof}

\begin{lem}\label{lem:almost}
If  $k \geq 4$, and  $\mathbf{x} \neq \mathbf{1}$  in (\ref{eqn:glincomb2}), then  $g(\mathbf{x}) = \pm\mathbf{x}_i$  for some $i \in \{1, \ldots, k, (1, 2), \ldots, (k - 1, k)\}$.
\end{lem}

\begin{proof}
Let  $k \geq 4$, and  $\mathbf{x} \neq \mathbf{1}$  in (\ref{eqn:glincomb2}).  By way of contradiction, suppose  $g(\mathbf{x}) = \pm0.5\mathbf{x}_{a, b} \pm0.5\mathbf{x}_{a, c} \pm0.5\mathbf{x}_b \pm0.5\mathbf{x}_c$  for some distinct  $a, b, c \in \{1, \ldots, k\}$.  By Lemma \ref{lem:cannotbe}, there exists some  $i \in \{1, \ldots, k\}$  such that  $g(\mathbf{x}_i) = \pm0.5\mathbf{x}_{a', b'} \pm0.5\mathbf{x}_{a', c'} \pm0.5\mathbf{x}_{b'} \pm0.5\mathbf{x}_{c'}$  for some distinct  $a', b', c' \in \{1, \ldots, k\}$.  Lemma \ref{lem:finalformforce} guarantees there will be another main effect sent to a similar form by  $g$, and Lemma \ref{lem:makeup} tells us it will be built from the same three distinct main effects and their three distinct $2$-factor interactions.  At most one main effect could be sent by  $g$  to a form other than that just described as noted in the proof of Lemma \ref{lem:finalformforce}.  Now we have at least one more main effect to consider, and it must be sent to a form similar to that given above and also built from the same three main effects and their three distinct $2$-factor interactions.  But now we have four main effects that are sent to linear combinations of six orthogonal vectors, and the six resulting $2$-factor interactions will necessarily also be sent by  $g$  to linear combinations of those same six orthogonal vectors (owing to the properties of the Hadamard product).  This means that the ten new vectors cannot all be orthogonal, which contradicts  $g \in G^{\rm LP}$.  Hence,  $g(\mathbf{x}) \neq \pm0.5\mathbf{x}_{a, b} \pm0.5\mathbf{x}_{a, c} \pm0.5\mathbf{x}_b \pm0.5\mathbf{x}_c$  for some distinct  $a, b, c \in \{1, \ldots, k\}$.
Now by Lemmas \ref{lem:finform} and \ref{lem:form}, we have  $g(\mathbf{x}) = \pm\mathbf{x}_i$  for some  $i \in \{1, \ldots, k, (1, 2), \ldots, (k - 1, k)\}$.
\end{proof}

\begin{lem}\label{lem:closingthegap}
Let  $k \geq 4$.  Then
\begin{equation*}
|G^{\rm LP}| \leq 2^k (k + 1)!.
\end{equation*}
\end{lem}

\begin{proof}
Let  $g \in G^{\rm LP}$  be arbitrary.  Note that because  $g$  preserves Hadamard products, knowing how it acts on the main effects will determine how it acts on all of  $\mathcal{B}$.  By Lemma \ref{lem:almost},   $g(\mathbf{x_1}) = \pm\mathbf{x}_i$  for  $i \in \{1, \ldots, k, (1, 2), \ldots, (k - 1, k)\}$.  Because  $g(\mathbf{x}_{1, 2})$  must be of a similar form, the possibilities for  $g(\mathbf{x}_2)$  are restricted depending upon  $g(\mathbf{x}_1)$.  If  $g(\mathbf{x}_1) = \pm\mathbf{x}_i$  for  $i \in \{1, \ldots, k\}$, then  $g(\mathbf{x}_2) = \pm\mathbf{x}_l$  for  $l \in \{1, \ldots, i - 1, i + 1, \ldots, k, (1, i), \ldots, (i - 1, i), (i, i + 1), \ldots, (i, k)\}$.  Otherwise,  $g(\mathbf{x}_1) = \pm\mathbf{x}_{i, j}$  for  $i < j$  and  $i, j \in \{1, \ldots, k\}$, so  $g(\mathbf{x}_2) = \pm\mathbf{x}_l$  for  $l \in \{i, j, (1, i), (1, j), \ldots, (i - 1, i), (i - 1, j), (i, i + 1), (i + 1, j), \ldots, (i, j - 1), (j - 1, j), (i, j + 1), (j, j + 1), \ldots, (i, k), (j, k)\}$.  To determine how many distinct possibilities exist, we shall consider four cases, based on the forms of  $g(\mathbf{x}_1)$  and  $g(\mathbf{x}_2)$, respectively.  \par
	(Case 1:  main effect, main effect)  Suppose  $g(\mathbf{x}_1) = \pm\mathbf{x}_i$  for  $i \in \{1, \ldots, k\}$  and  $g(\mathbf{x}_2) = \pm\mathbf{x}_l$  for  $l \in \{1, \ldots, i - 1, i + 1, \ldots, k\}$.  Then there are  $2 k$  possibilities for  $g(\mathbf{x}_1)$  and  $2 (k - 1)$  for  $g(\mathbf{x}_2)$.  All of the  $(k - 2)$  remaining main effects must be sent to plus or minus the other  $(k - 2)$  main effects.  That is, there are  $(2 k) (2 (k - 1)) (2^{k - 2} (k - 2)!) = 2^k k!$  distinct possibilities.  \par
	(Case 2:  main effect, $2$-factor interaction)  Suppose  $g(\mathbf{x}_1) = \pm\mathbf{x}_i$  for  $i \in \{1, \ldots, k\}$  and  $g(\mathbf{x}_2) = \pm\mathbf{x}_l$  for  $l \in \{(1, i), \ldots, (i - 1, i), (i, i + 1), \ldots, (i, k)\}$.  Then there are  $2 k$  possibilities for  $g(\mathbf{x}_1)$  and  $2 (k - 1)$  for  $g(\mathbf{x}_2)$.  All of the  $(k - 2)$  remaining main effects must be sent to plus or minus the other  $(k - 2)$  viable $2$-factor interactions.  That is, there are  $(2 k) (2 (k - 1)) (2^{k - 2} (k - 2)!) = 2^k k!$  distinct possibilities.  \par
	(Case 3:  $2$-factor interaction, main effect)  Suppose  $g(\mathbf{x}_1) = \pm\mathbf{x}_{i, j}$  for  $i < j$  and  $i, j \in \{1, \ldots, k\}$  and  $g(\mathbf{x}_2) = \pm\mathbf{x}_l$  for  $l \in \{i, j\}$.  Then there are  $2 \binom{k}{2}$  possibilities for  $g(\mathbf{x}_1)$  and  $2 (2)$  for  $g(\mathbf{x}_2)$.  All of the  $(k - 2)$  remaining main effects must be sent to plus or minus the other  $(k - 2)$  viable $2$-factor interactions.  That is, there are  $(2 \binom{k}{2}) (2(2)) (2^{k - 2} (k - 2)!) = 2^k \binom{k}{2} (2) (k - 2)!$  distinct possibilities.  \par
	(Case 4:  $2$-factor interaction, $2$-factor interaction)  Suppose  $g(\mathbf{x}_1) = \pm\mathbf{x}_{i, j}$  for  $i < j$  and  $i, j \in \{1, \ldots, k\}$  and  $g(\mathbf{x}_2) = \pm\mathbf{x}_l$  for  $l \in \{(1, i), (1, j), \ldots, (i - 1, i), (i - 1, j), (i, i + 1), (i + 1, j), \ldots, (i, j - 1), (j - 1, j), (i, j + 1), (j, j + 1), \ldots, (i, k), (j, k)\}$.  Then there are  $2 \binom{k}{2}$  possibilities for  $g(\mathbf{x}_1)$  and  $2 (2 k - 4)$  for  $g(\mathbf{x}_2)$.  All of the  $(k - 2)$  remaining main effects must be sent to plus or minus the other  $(k - 3)$  viable $2$-factor interactions and the lone viable main effect  $\mathbf{x}_{\{i, j\} \cap l}$.  That is, there are  $(2 \binom{k}{2}) (2 (2 k - 4)) (2^{k - 2} (k - 2)!) = 2^k \binom{k}{2} (2 k - 4) (k - 2)!$  distinct possibilities.  \par
	Therefore, the total number of possibilities for all cases is
$$
2^k k!  +  2^k k! + 2^k \binom{k}{2} (2) (k - 2)! + 2^k \binom{k}{2} (2 k - 4) (k - 2)!= 2^k (k + 1)!.\\
$$
Thus,  $g$  is one of at most  $2^k (k + 1)!$  elements in  $G^{\rm LP}$.
\end{proof}

\begin{proof}[{\bf Proof of Theorem \ref{thm:main}}]

Let  $R = \langle\rho_1, \ldots, \rho_k\rangle$  where  $\rho_i$  acts on the full factorial design by sending  $(\mathbf{x}_1, \ldots, \mathbf{x}_i, \ldots, \mathbf{x}_k)$  to  $(\mathbf{x}_{1,i}, \ldots, \mathbf{x}_i, \ldots, \mathbf{x}_{i,k})$  for  $i = 1, \ldots, k$.  Note that elements of  $R$  preserve the full factorial design as well as  ${\rm Row}(\mathbf{M})$, so  $R \leq G$.  Furthermore, for  $i = 1, \ldots, k$,  $\rho_i^{-1} = \rho_i$.  For any distinct  $i, j \in \{1, \ldots, k\}$,  $\rho_i \rho_j \rho_i$  simply permutes  $\mathbf{x}_i$  and  $\mathbf{x}_j$  within the full factorial design, so clearly  $S_k \leq R$.  Now we see that  $\rho_j \rho_i S_k = \rho_i S_k$  for some distinct  $i, j \in \{1, \ldots, k\}$, so there are exactly  $k + 1$  left cosets of  $S_k$  within  $R$, and together these constitute $R$.  Hence,  $R \cong S_{k + 1}$.  Letting  $\phi \in S_2^k$  be arbitrary, we note that for any  $i = 1, \ldots, k$,  $\rho_i^{-1} = \rho_i$,  $\rho_i^{-1} \phi \rho_i = \phi'$  where  $\phi' \in S_2^k$, and  $\rho_i^{-1} \phi' \rho_i = \phi$.  Together with this information, the fact that $S_2^k \trianglelefteq S_2\wr S_k$ makes it clear that  $S_2^k \trianglelefteq S_2^k \rtimes S_{k + 1}$.  Now an isomorphic copy of $S_2^k \rtimes S_{k + 1}$ is contained in $G^{\rm LP}$, and  $|S_2^k \rtimes S_{k + 1}| = |S_2^k| |S_{k + 1}| = 2^k (k + 1)!$. Hence by Lemma \ref{lem:closingthegap},  $G^{\rm LP}$  must be isomorphic to  $S_2^k \rtimes S_{k + 1}$.
\end{proof}

For  $k = 3$, Theorem \ref{thm:main} does not hold.  To see this, consider the permutation  $g \in G^{\rm LP}$  such that
\begin{eqnarray*}
g(\mathbf{x}_1) & = & 0.5 \mathbf{x}_{1, 2} + 0.5 \mathbf{x}_{1, 3} + 0.5 \mathbf{x}_2 - 0.5 \mathbf{x}_3 \\
g(\mathbf{x}_2) & = & 0.5 \mathbf{x}_{1, 2} + 0.5 \mathbf{x}_{1, 3} - 0.5 \mathbf{x}_2 + 0.5 \mathbf{x}_3 \\
g(\mathbf{x}_3) & = & \mathbf{x}_1.
\end{eqnarray*}
Because this permutation sends main effects to forms other than those which are viable for  $k \geq 4$, we conclude  $|G^{\rm LP}| > 2^3 (3 + 1)! = 192$.  This observation is corroborated by the Geyer~\cite{Geyer2014b} algorithm and GAP~\cite{GAP2013}, which prove that in this case  $|G^{\rm LP}| = 1152$, and  $G^{\rm LP} \cong (S_4 \times S_4) \rtimes S_2$.
\section{The significance of the results and future research}
Firstly, the theoretical results of this paper confirm the computational observations of Geyer~\cite{Geyer2014b}.
This is a good check on the validity of the algorithm in Geyer~\cite{Geyer2014b} for finding $G^{\rm LP}$.
Also, the algorithm in Geyer~\cite{Geyer2014b} stalls as the problem size gets larger. Here, we have solved this 
problem for any formulation that has the same set of variables and LP relaxation feasible set as that of the Bulutoglu and Margot formulation
 of all strength $1$ and $2$ cases regardless of the problem size.

Exploiting the previously unknown symmetries in $G^{\rm LP}\cong S_2^k \rtimes S_{k + 1}$ for finding orthogonal arrays has already made it
possible to find OA$(160,9,2,4)$ and OA$(176,9,2,4)$ and prove the non-existence of   OA$(160,10,2,4)$ and OA$(176,10,2,4)$, see Bulutoglu and Ryan~\cite{Bulutoglu2016}.
The concept of OD-equivalence rejection of $\oant$ introduced in~\cite{Bulutoglu2016} was necessary in obtaining the results therein. OD-equivalence rejection is equivalent to exploiting $S_2^k \rtimes S_{k + 1}$ by keeping only one representative $\oant$ under the action of this group.

For $s=2$ it is easy to see that an isomorphic copy of $S_2 \wr S_k$  is always a subgroup of  $G^{\rm LP}$, and an isomorphic copy $S_2^k \rtimes S_{k + 1}$  is a subgroup when  $t$  is even. Based on the computational results in Geyer~\cite{Geyer2014b}, we conjecture that $G^{\rm LP} \cong S_2\wr S_k$ when $t$ is odd, and $G^{\rm LP} \cong S_2\rtimes S_{k+1}$ otherwise.
   Also, for $s>2$ it is easy to see that an isomorphic copy of $S_s\wr S_k$ is a subgroup of $G^{\rm LP}$. Computational results in Geyer~\cite{Geyer2014b} lead to the conjecture that $G^{\rm LP} \cong S_s\wr S_k$ for $s>2$. We propose settling these conjectures as open problems. All of our conjectures pertain to the Bulutoglu and Margot  formulation or any formulation obtained from the Bulutoglu and Margot formulation by applying 
the elementary row operations to its equality constraint matrix.
\section*{Acknowledgements}
The constructive comments from an anonymous Referee improved the presentation of our results  substantially. This research was supported by the AFOSR grant F4FGA04013J001.  The views expressed in this article are those of the authors and do not reflect the official policy or position of the United States Air Force, Department of Defense, or the U.S. Government.\\
\bibliographystyle{spmpsci}
\bibliography{bibliography}

\begin{thebibliography}{10}
\providecommand{\url}[1]{{#1}}
\providecommand{\urlprefix}{URL }
\expandafter\ifx\csname urlstyle\endcsname\relax
  \providecommand{\doi}[1]{DOI~\discretionary{}{}{}#1}\else
  \providecommand{\doi}{DOI~\discretionary{}{}{}\begingroup
  \urlstyle{rm}\Url}\fi

\bibitem{Appa2006}
Appa, G., Magos, D., Mourtos, I.: On multi-index assignment polytopes.
\newblock Linear Algebra and its Applications \textbf{416}(2–3), 224--241
  (2006).
\newblock \doi{https://doi:10.1016/j.laa.2005.11.009}

\bibitem{Bulutoglu2008}
Bulutoglu, D.A., Margot, F.: Classification of orthogonal arrays by integer
  programming.
\newblock Journal of Statistical Planning and Inference \textbf{138}(3),
  654--666 (2008)

\bibitem{Bulutoglu2016}
Bulutoglu, D.A., Ryan, K.J.: Integer programming for classifying orthogonal
  arrays.
\newblock Australasian Journal of Combinatorics \textbf{70}(3), in press (2018)

\bibitem{GAP2013}
The GAP~Group: {GAP -- Groups, Algorithms, and Programming, Version 4.6.4}
  (2013).
\newblock \urlprefix\url{\texttt({https://www.gap-system.org})}

\bibitem{Geyer2014b}
Geyer, A.J.: Different formulations of the orthogonal array problem and their
  symmetries.
\newblock {PhD} dissertation, Air Force Institute of Technology (2014)

\bibitem{Geyer2014}
Geyer, A.J., Bulutoglu, D.A., Rosenberg, S.J.: The {LP} relaxation orthogonal
  array polytope and its permutation symmetries.
\newblock Journal of Combinatorial Mathematics and Combinatorial Computing
  \textbf{91}, 165--176 (2014)

\bibitem{Liberti2012}
Liberti, L.: Reformulations in mathematical programming: Automatic symmetry
  detection and exploitation.
\newblock Mathematical Programming Series A \textbf{131}(1-2), 273--304 (2012)

\bibitem{Margot2003b}
Margot, F.: Exploiting orbits in symmetric {ILP}.
\newblock Mathematical Programming Series B \textbf{98}(1-3), 3--21 (2003)

\bibitem{Margot2003a}
Margot, F.: Small covering designs by branch-and-cut.
\newblock Mathematical Programming Series B \textbf{94}(2-3), 207--220 (2003)

\bibitem{Margot2007}
Margot, F.: Symmetric {ILP}: Coloring and small integers.
\newblock Discrete Optimization \textbf{4}(1), 40--62 (2007)

\bibitem{Margot2010}
Margot, F.: Symmetry in integer linear programming.
\newblock In: M.~Junger, T.~Liebling, D.~Naddef, G.L. Nemhauser, W.~Pulleybank,
  G.~Reinelt, G.~Rinaldi, L.~Wolsey (eds.) 50 Years of Integer Programming
  1958-2008, pp. 647--686. Springer-Verlag, Berlin, Heidelberg, Germany (2010)

\bibitem{R2013}
{R Core Team}: R: A Language and Environment for Statistical Computing.
\newblock R Foundation for Statistical Computing, Vienna, Austria (2013).
\newblock \urlprefix\url{\texttt({http://www.R-project.org/})}

\bibitem{Rosenberg1995}
Rosenberg, S.J.: A large index theorem for orthogonal arrays, with bounds.
\newblock Discrete Mathematics \textbf{137}(1), 315--318 (1995)

\bibitem{Rotman1994}
Rotman, J.J.: An Introduction to the Theory of Groups, 4 edn.
\newblock Springer-Verlag, New York, NY, USA (1994)

\bibitem{Stufken2007}
Stufken, J., Tang, B.: Complete enumeration of two-level orthogonal arrays of
  strength $d$ with $d+2$ constraints.
\newblock Annals of Statistics \textbf{35}(2), 793--814 (2007).
\newblock \doi{https://dx.doi.org/10.1214/009053606000001325}

\end{thebibliography}
\begin{appendix}
\section{R code and the Cases}
\footnotesize
\singlespacing
\begin{verbatim}
################################################################################
# Function-  expandcases - generates linear combinations with positive 1st term
# Input(s)-  lst - binary representation of main effects and interactions
# Output(s)- newlst - linear combinations with positive 1st term
################################################################################
expandcases<-function(lst){
ll<-length(lst)
newlst<-list()
r<-1
aa<-as.matrix(expand.grid(c(1,-1),c(1,-1),c(1,-1)))
aa<-cbind(1,aa)
for (j in 1:ll){
for (i in 1:8){
aa2<-rbind(aa[i,],lst[[j]])
 dimnames(aa2)[[2]]<-NULL
newlst[r]<-list(aa2)
r<-r+1
}}
return(newlst)
}
################################################################################
# Function-  allcheckbinvector - checks viability of linear combinations
# Input(s)-  lst - all linear combinations to be checked
# Output(s)- displays 1s for viable combinations and 0s otherwise
################################################################################
allcheckbinvector<-function(lst){
l<-length(lst)
for (i in 1:l){
print(checkbinaryvector(lst[[i]]))
}
}
################################################################################
# Function-  checkbinaryvector - checks viability of a linear combination
# Input(s)-  newcik - the linear combination to be checked
# Output(s)- 1 if combination is viable, 0 otherwise
################################################################################
checkbinaryvector<-function(newcik){
newcikmat<-newcik[-1,]
newcikcoff<-newcik[1,]
pp<-dim(newcikmat)
cols<-pp[2]
pp<-pp[1]
full<-c("expand.grid(c(1,-1)")
for (j in 1:(pp-1)){
full<-paste(full,",c(1,-1)")
}
full<-paste(full,")")
full<-parse(text = full)
full<-eval(full)
temp<-genprod(full,newcikmat[,1])
for (i in 2:cols){
temp<-cbind(temp,genprod(full,newcikmat[,i]))
}
finvec<-as.matrix(temp[,1]*newcikcoff[1])
for (i in 2:cols){
finvec<-finvec+newcikcoff[i]*temp[,i]
}
finvec<-as.matrix(as.integer(finvec/2))
if(min(finvec)==-1 & max(finvec)==1){
return(1)} else{return(0)}
}
################################################################################
# Function-  genprod - computes Hadamard product
# Input(s)-  full  2 level (+/-1) full factorial design
#            tt1 - indicator of which main effects are to be multiplied
# Output(s)- as.matrix(outfullcheck) - +/-1 form of main effect or interaction
################################################################################
genprod<-function(full,tt1){
outfullcheck<-full[,1]^(tt1[1])
pp<-length(tt1)
for (j in 1:(pp-1)){
outfullcheck<-outfullcheck*full[,(j+1)]^(tt1[(j+1)])
}
return(as.matrix(outfullcheck))
}
################################################################################
# Cases with 1 two-factor interaction
################################################################################
a12.1.2.3<-cbind(c(1,1,0),c(1,0,0),c(0,1,0),c(0,0,1))
c12.1.2.3<-expandcases(list(a12.1.2.3))
allcheckbinvector(c12.1.2.3)
a12.1.3.4<-cbind(c(1,1,0,0),c(1,0,0,0),c(0,0,1,0),c(0,0,0,1))
c12.1.3.4<-expandcases(list(a12.1.3.4))
allcheckbinvector(c12.1.3.4)
a12.3.4.5<-cbind(c(1,1,0,0,0),c(0,0,1,0,0),c(0,0,0,1,0),c(0,0,0,0,1))
c12.3.4.5<-expandcases(list(a12.3.4.5))
allcheckbinvector(c12.3.4.5)
################################################################################
# Cases with 2 two-factor interactions
################################################################################
a12.13<-cbind(c(1,1,0,0,0,0),c(1,0,1,0,0,0))
a12.34<-cbind(c(1,1,0,0,0,0),c(0,0,1,1,0,0))
a12.13.1<-cbind(a12.13,c(1,0,0,0,0,0))
a12.13.2<-cbind(a12.13,c(0,1,0,0,0,0))
a12.13.4<-cbind(a12.13,c(0,0,0,1,0,0))
a12.34.1<-cbind(a12.34,c(1,0,0,0,0,0))
a12.34.5<-cbind(a12.34,c(0,0,0,0,1,0))
a12.13.1.2<-cbind(a12.13.1,c(0,1,0,0,0,0))
a12.13.1.2<-a12.13.1.2[c(1:3),]
c12.13.1.2<-expandcases(list(a12.13.1.2))
allcheckbinvector(c12.13.1.2)
a12.13.1.4<-cbind(a12.13.1,c(0,0,0,1,0,0))
a12.13.1.4<-a12.13.1.4[c(1:4),]
c12.13.1.4<-expandcases(list(a12.13.1.4))
allcheckbinvector(c12.13.1.4)
a12.13.2.3<-cbind(a12.13.2,c(0,0,1,0,0,0))
a12.13.2.3<-a12.13.2.3[c(1:3),]
c12.13.2.3<-expandcases(list(a12.13.2.3))
allcheckbinvector(c12.13.2.3)
a12.13.2.4<-cbind(a12.13.2,c(0,0,0,1,0,0))
a12.13.2.4<-a12.13.2.4[c(1:4),]
c12.13.2.4<-expandcases(list(a12.13.2.4))
allcheckbinvector(c12.13.2.4)
a12.13.4.5<-cbind(a12.13.4,c(0,0,0,0,1,0))
a12.13.4.5<-a12.13.4.5[c(1:5),]
c12.13.4.5<-expandcases(list(a12.13.4.5))
allcheckbinvector(c12.13.4.5)
a12.34.1.2<-cbind(a12.34.1,c(0,1,0,0,0,0))
a12.34.1.2<-a12.34.1.2[c(1:4),]
c12.34.1.2<-expandcases(list(a12.34.1.2))
allcheckbinvector(c12.34.1.2)
a12.34.1.3<-cbind(a12.34.1,c(0,0,1,0,0,0))
a12.34.1.3<-a12.34.1.3[c(1:4),]
c12.34.1.3<-expandcases(list(a12.34.1.3))
allcheckbinvector(c12.34.1.3)
a12.34.1.5<-cbind(a12.34.1,c(0,0,0,0,1,0))
a12.34.1.5<-a12.34.1.5[c(1:5),]
c12.34.1.5<-expandcases(list(a12.34.1.5))
allcheckbinvector(c12.34.1.5)
a12.34.5.6<-cbind(a12.34.5,c(0,0,0,0,0,1))
a12.34.5.6<-a12.34.5.6[c(1:6),]
c12.34.5.6<-expandcases(list(a12.34.5.6))
allcheckbinvector(c12.34.5.6)
################################################################################
# Cases with 3 two-factor interactions
################################################################################
a12.13<-cbind(c(1,1,0,0,0,0,0),c(1,0,1,0,0,0,0))
a12.34<-cbind(c(1,1,0,0,0,0,0),c(0,0,1,1,0,0,0))
a12.13.14<-cbind(a12.13,c(1,0,0,1,0,0,0))
a12.13.23<-cbind(a12.13,c(0,1,1,0,0,0,0))
a12.13.24<-cbind(a12.13,c(0,1,0,1,0,0,0))
a12.13.45<-cbind(a12.13,c(0,0,0,1,1,0,0))
a12.34.56<-cbind(a12.34,c(0,0,0,0,1,1,0))
a12.13.14.1<-cbind(a12.13.14,c(1,0,0,0,0,0,0))
a12.13.14.1<-a12.13.14.1[c(1:4),]
c12.13.14.1<-expandcases(list(a12.13.14.1))
allcheckbinvector(c12.13.14.1)
a12.13.14.2<-cbind(a12.13.14,c(0,1,0,0,0,0,0))
a12.13.14.2<-a12.13.14.2[c(1:4),]
c12.13.14.2<-expandcases(list(a12.13.14.2))
allcheckbinvector(c12.13.14.2)
a12.13.14.5<-cbind(a12.13.14,c(0,0,0,0,1,0,0))
a12.13.14.5<-a12.13.14.5[c(1:5),]
c12.13.14.5<-expandcases(list(a12.13.14.5))
allcheckbinvector(c12.13.14.5)
a12.13.23.1<-cbind(a12.13.23,c(1,0,0,0,0,0,0))
a12.13.23.1<-a12.13.23.1[c(1:3),]
c12.13.23.1<-expandcases(list(a12.13.23.1))
allcheckbinvector(c12.13.23.1)
a12.13.23.4<-cbind(a12.13.23,c(0,0,0,1,0,0,0))
a12.13.23.4<-a12.13.23.4[c(1:4),]
c12.13.23.4<-expandcases(list(a12.13.23.4))
allcheckbinvector(c12.13.23.4)
a12.13.24.1<-cbind(a12.13.24,c(1,0,0,0,0,0,0))
a12.13.24.1<-a12.13.24.1[c(1:4),]
c12.13.24.1<-expandcases(list(a12.13.24.1))
allcheckbinvector(c12.13.24.1)
a12.13.24.3<-cbind(a12.13.24,c(0,0,1,0,0,0,0))
a12.13.24.3<-a12.13.24.3[c(1:4),]
c12.13.24.3<-expandcases(list(a12.13.24.3))
allcheckbinvector(c12.13.24.3)
a12.13.24.5<-cbind(a12.13.24,c(0,0,0,0,1,0,0))
a12.13.24.5<-a12.13.24.5[c(1:5),]
c12.13.24.5<-expandcases(list(a12.13.24.5))
allcheckbinvector(c12.13.24.5)
a12.13.45.1<-cbind(a12.13.45,c(1,0,0,0,0,0,0))
a12.13.45.1<-a12.13.45.1[c(1:5),]
c12.13.45.1<-expandcases(list(a12.13.45.1))
allcheckbinvector(c12.13.45.1)
a12.13.45.2<-cbind(a12.13.45,c(0,1,0,0,0,0,0))
a12.13.45.2<-a12.13.45.2[c(1:5),]
c12.13.45.2<-expandcases(list(a12.13.45.2))
allcheckbinvector(c12.13.45.2)
a12.13.45.5<-cbind(a12.13.45,c(0,0,0,0,1,0,0))
a12.13.45.5<-a12.13.45.5[c(1:5),]
c12.13.45.5<-expandcases(list(a12.13.45.5))
allcheckbinvector(c12.13.45.5)
a12.13.45.6<-cbind(a12.13.45,c(0,0,0,0,0,1,0))
a12.13.45.6<-a12.13.45.6[c(1:6),]
c12.13.45.6<-expandcases(list(a12.13.45.6))
allcheckbinvector(c12.13.45.6)
a12.34.56.1<-cbind(a12.34.56,c(1,0,0,0,0,0,0))
a12.34.56.1<-a12.34.56.1[c(1:6),]
c12.34.56.1<-expandcases(list(a12.34.56.1))
allcheckbinvector(c12.34.56.1)
a12.34.56.7<-cbind(a12.34.56,c(0,0,0,0,0,0,1))
c12.34.56.7<-expandcases(list(a12.34.56.7))
allcheckbinvector(c12.34.56.7)
################################################################################
# Cases with 4 two-factor interactions
################################################################################
a12.13<-cbind(c(1,1,0,0,0,0,0,0),c(1,0,1,0,0,0,0,0))
a12.34<-cbind(c(1,1,0,0,0,0,0,0),c(0,0,1,1,0,0,0,0))
a12.13.14<-cbind(a12.13,c(1,0,0,1,0,0,0,0))
a12.13.23<-cbind(a12.13,c(0,1,1,0,0,0,0,0))
a12.13.24<-cbind(a12.13,c(0,1,0,1,0,0,0,0))
a12.13.45<-cbind(a12.13,c(0,0,0,1,1,0,0,0))
a12.34.56<-cbind(a12.34,c(0,0,0,0,1,1,0,0))
a12.13.14.15<-cbind(a12.13.14,c(1,0,0,0,1,0,0,0))
a12.13.14.15<-a12.13.14.15[c(1:5),]
c12.13.14.15<-expandcases(list(a12.13.14.15))
allcheckbinvector(c12.13.14.15)
a12.13.14.23<-cbind(a12.13.14,c(0,1,1,0,0,0,0,0))
a12.13.14.23<-a12.13.14.23[c(1:4),]
c12.13.14.23<-expandcases(list(a12.13.14.23))
allcheckbinvector(c12.13.14.23)
a12.13.14.25<-cbind(a12.13.14,c(0,1,0,0,1,0,0,0))
a12.13.14.25<-a12.13.14.25[c(1:5),]
c12.13.14.25<-expandcases(list(a12.13.14.25))
allcheckbinvector(c12.13.14.25)
a12.13.14.56<-cbind(a12.13.14,c(0,0,0,0,1,1,0,0))
a12.13.14.56<-a12.13.14.56[c(1:6),]
c12.13.14.56<-expandcases(list(a12.13.14.56))
allcheckbinvector(c12.13.14.56)
a12.13.23.45<-cbind(a12.13.23,c(0,0,0,1,1,0,0,0))
a12.13.23.45<-a12.13.23.45[c(1:5),]
c12.13.23.45<-expandcases(list(a12.13.23.45))
allcheckbinvector(c12.13.23.45)
a12.13.24.35<-cbind(a12.13.24,c(0,0,1,0,1,0,0,0))
a12.13.24.35<-a12.13.24.35[c(1:5),]
c12.13.24.35<-expandcases(list(a12.13.24.35))
allcheckbinvector(c12.13.24.35)
a12.13.24.56<-cbind(a12.13.24,c(0,0,0,0,1,1,0,0))
a12.13.24.56<-a12.13.24.56[c(1:6),]
c12.13.24.56<-expandcases(list(a12.13.24.56))
allcheckbinvector(c12.13.24.56)
a12.13.45.56<-cbind(a12.13.45,c(0,0,0,0,1,1,0,0))
a12.13.45.56<-a12.13.45.56[c(1:6),]
c12.13.45.56<-expandcases(list(a12.13.45.56))
allcheckbinvector(c12.13.45.56)
a12.13.45.67<-cbind(a12.13.45,c(0,0,0,0,0,1,1,0))
a12.13.45.67<-a12.13.45.67[c(1:7),]
c12.13.45.67<-expandcases(list(a12.13.45.67))
allcheckbinvector(c12.13.45.67)
a12.34.56.78<-cbind(a12.34.56,c(0,0,0,0,0,0,1,1))
c12.34.56.78<-expandcases(list(a12.34.56.78))
allcheckbinvector(c12.34.56.78)
\end{verbatim}
\end{appendix}

\end{document}